\documentclass[reqno,11pt]{article}
\usepackage{a4wide,color,eucal,enumerate,mathrsfs}
\usepackage[normalem]{ulem}
\usepackage{amsmath,amssymb,epsfig,bbm}
\usepackage{pdfsync}
\numberwithin{equation}{section}
\usepackage[pdfborder={0 0 0}]{hyperref}

\newcommand{\R}{\mathbb{R}}

\newcommand{\mm}{{\mbox{\boldmath$m$}}}

\newcommand{\sfd}{{\sf d}}

\newcommand{\restr}[1]{\lower3pt\hbox{$|_{#1}$}}

\newcommand{\eps}{\varepsilon}  
\newcommand{\nchi}{{\raise.3ex\hbox{$\chi$}}}

\newenvironment{proof}{\removelastskip\par\medskip   % inizio e fine dimostrazione
\noindent{\em Proof.}
\rm}{\penalty-20\null\hfill$\square$\par\medbreak}

\newtheorem{theorem}{Theorem}[section]

\newtheorem{corollary}[theorem]{Corollary}
\newtheorem{lemma}[theorem]{Lemma}
\newtheorem{proposition}[theorem]{Proposition}
\newtheorem{definition}[theorem]{Definition}

\newtheorem{remark}[theorem]{Remark}

\newcommand{\Minl}[1]{{\cal M}_-(#1)}
\newcommand{\Minu}[1]{{\cal M}_+(#1)}
\newcommand{\RelMinl}[1]{{\cal M}^*_-(#1)}
\newcommand{\RelMinu}[1]{{\cal M}^*_+(#1)}
\newcommand{\RelMin}[1]{{\cal M}^*(#1)}
\newcommand{\Per}[1]{{\rm Per}(#1)}
\newcommand{\Var}[1]{{\rm Var}(#1)}
\newcommand{\Haus}[1]{{\mathscr H}^{#1}}     % Misura di Hausdorff
\newcommand{\Hausph}[1]{{\mathscr S}^{#1}}     % Misura di Hausdorff sferica
\newcommand{\Leb}[1]{{\mathscr L}^{#1}}      % Misura di Lebesgue

\newcommand{\Slope}[1]{{\rm lip}(#1)}

\newcommand{\fr}{\hfill$\blacksquare$}   %quadratino nero alla fine del remark, se non vi piace, la cosa migliore e' `svuotare' la macro, cosi' non bisogna intervenire sul testo
\newcommand{\res}{\mathop{\hbox{\vrule height 7pt width .5pt depth 0pt
\vrule height .5pt width 6pt depth 0pt}}\nolimits} % macro per la restrizione

\newcommand{\RCD}{\mathrm{RCD}}

\renewcommand{\mm}{\mathfrak m}

\begin{document}

\title{Perimeter as relaxed Minkowski content \\ in metric measure spaces}

\author{Luigi Ambrosio\
   \thanks{Scuola Normale Superiore, Pisa. email: \textsf{luigi.ambrosio@sns.it}}
   \and
   Simone Di Marino\
 \thanks{Laboratoire de Math. d'Orsay, Univ. Paris-Sud, CNRS, Univ. Paris-Saclay. email: \textsf{sthew87@gmail.com}}
 \and
   Nicola Gigli\
 \thanks{SISSA, Trieste. email: \textsf{ngigli@sissa.it}}
   }

\maketitle
\begin{abstract} In this note we prove that on general metric measure spaces the perimeter is equal to the relaxation of the Minkowski content 
w.r.t.\ convergence in measure.
\end{abstract}

%\tableofcontents

\section{Introduction}

In a metric measure space $(X,\sfd,\mm)$, the upper and lower Minkowski contents are respectively defined by
$$%\begin{equation}\label{eq:111}
\Minl{A}:=\liminf_{r\downarrow 0}\frac{\mm(A^r)-\mm(A)}{r},\qquad
\Minu{A}:=\limsup_{r\downarrow 0}\frac{\mm(A^r)-\mm(A)}{r}
$$%\end{equation}
for Borel sets $A$ with finite $\mm$-measure. The Minkowski contents, even in their non-infinitesimal versions, appear in many areas,
as in the theory of concentration of measures and isoperimetric inequalities (see for instance \cite{LED}, \cite{CavallettiMondinoa}, 
\cite{CavallettiMondinob} and the references therein) and the theory of random closed sets \cite{ACV}. 

For sufficiently nice metric measure structures, the relations between Minkowski content and perimeter are well-known, see
for instance \S2.13 in \cite{AFP} and \S14.2 in \cite{Burago}.
Aim of this note is the investigation of more precise relations between the Minkowski content and the perimeter, as defined
in the theory of $BV$ functions in metric measure spaces. In particular we prove in Lemma~\ref{lem:main1} and Theorem~\ref{thm:main} that the lower semicontinuous envelope
w.r.t. $L^1(X,\mm)$ convergence of the Minkowski contents ${\cal M}_\pm(A)$ is equal to the perimeter $\Per{A}$. As a byproduct, we can prove that
in metric measure spaces with finite $\mm$-measure the Cheeger constant
$$
\gamma:=\inf\left\{\frac{\Minu{A}}{\mm(A)}:\ 0<\mm(A)\leq\frac{\mm(X)}{2}\right\}
$$
can be equivalently defined replacing $\Minu{A}$ with $\Minl{A}$, or with $\Per{A}$.

Another consequence of our result is that whenever one wants to establish, on a given space, an isoperimetric inequality of the form
\[
\mm(A)\leq f(\Per{A})
\]
for some continuous non-decreasing function $f:\R^+\to \R^+$, it is sufficient to prove the easier
\[
\mm(A)\leq f(\Minu{A}).
\]
We remark that not only it is easily seen that $\Minu{A}\geq \Per A$ (see also the proof of Theorem \ref{thm:main}), but also that the Minkowski content is a quantity that by nature is sometimes handled better than the perimeter in estimates involving the geometry of the space. An example in this direction is the recent paper \cite{CavallettiMondinoa} by Cavalletti-Mondino, which motivated our study (see also the work in progress \cite{CavallettiMondinoc}, which contains
results similar to ours, under curvature assumptions).

Another goal of the paper is a closer investigation of the coarea formula and of the ``generic'' properties of superlevel sets
of Lipschitz functions. In Euclidean and other nice spaces, the combination of the Fleming-Rishel formula 
(involving the perimeter of superlevel sets) with the coarea formula for Lipschitz maps (involving the 
Hausdorff measure of level sets) provides many useful informations, even on the level sets, as illustrated in Remark~\ref{It_is_not_trivial}. 
Under a suitable regularity assumption \eqref{eq:13} on the metric measure structure, fulfilled in all spaces $RCD(K,\infty)$ of \cite{AGS11b}, we provide
in Proposition~\ref{prop:coarea} a metric counterpart of this, involving the Minkowski contents. Finally, we are able to make
a more detailed analysis for level sets of distance functions and we conclude the paper pointing out a few open questions.

\section{Basic setting and preliminaries}

Throughout this paper $(X,\sfd)$ is a metric space and $\mm$ is a nonnegative and $\sigma$-additive measure on its
Borel $\sigma$-algebra; we always assume that $\mm$ is finite on bounded Borel sets. In particular, $\mm$ is $\sigma$-finite.

In the metric space $(X,\sfd)$ we denote by $B_r(x)$ the open ball with center $x$ and radius $r$. We denote by ${\rm Lip}(f)$ the
Lipschitz constant of a Lipschitz function $f:X\to\R$ and we will often use the distance function
$$
\sfd_A(x):=\inf_{y\in A}\sfd(x,y)
$$
from a nonempty set $A$, whose Lipschitz constant is less than 1. The slope $\Slope{f}$ (also
called local Lipschitz constant) of $f:X\to\R$ is defined by
$$
\Slope{f}(x):=\limsup_{y\to x}\frac{|f(y)-f(x)|}{\sfd(y,x)},
$$
with the convention $\Slope{f}(x)=0$ if $x$ is an isolated point. 

We denote by $\chi_A:X\to\{0,1\}$ the characteristic function of a set A and we say that $A_h\to A$ in $\mm$-measure if
$\int_X|\chi_{A_h}-\chi_A|\,d\mm\to 0$ (equivalently, $\mm(A_h\Delta A)\to 0$). For any nonempty 
set $A\subset X$ and any $r>0$ we define the open $r$-enlargement $A^r$ of $A$ by
$$
A^r:=\left\{x\in X:\ \sfd_A(x)<r\right\}.
$$
Notice that $A^r=(\overline{A})^r$, and that the triangle inequality gives the semigroup inclusion
\begin{equation}\label{eq:0}
(A^s)^t\subset A^{s+t}\qquad s,\,t>0.
\end{equation}

For any Borel set $A\subset X$ with $\mm(A)<\infty$ we define the upper and lower Minkowski contents by
\[
\Minl{A}:=\liminf_{r\downarrow 0}\frac{\mm(A^r)-\mm(A)}{r},\qquad
\Minu{A}:=\limsup_{r\downarrow 0}\frac{\mm(A^r)-\mm(A)}{r}.
\]
Obviously one has $0\leq \Minl{A}\leq\Minu{A}\leq\infty$. In addition, $\Minl{A}=\infty$ if $\mm(\overline{A}\setminus A)>0$, hence
a natural domain for the Minkowski content is the class of essentially closed sets. Notice also that $\Minl{A}<\infty$ implies $\mm(A^r\setminus A)\to 0$
as $r\downarrow 0$. 

We define the relaxed Minkowski contents $\RelMinl{A}$, $\RelMinu{A}$ of $A$ by
\[
\begin{split}
\RelMinl{A}&:=\inf\left\{\liminf\limits_{h\to\infty}\Minl{A_h}:\ \text{$A_h\to A$ in $\mm$-measure}\right\},\\%&\cr &\cr
\RelMinu{A}&:=\inf\left\{\liminf\limits_{h\to\infty}\Minu{A_h}:\ \text{$A_h\to A$ in $\mm$-measure}\right\}.
\end{split}
\]

It is obvious that $\RelMinl{A}\leq\RelMinu{A}$. The following elementary lemma shows that equality holds, so from now on we will also use
the notation $\RelMin{A}$ for their common value.

\begin{lemma} \label{lem:main1} For any Borel set $A\subset X$ one has
\begin{equation}\label{eq:4}
\Minl{A}\geq\RelMinu{A}.
\end{equation}
In particular $\RelMinl{A}=\RelMinu{A}$.
\end{lemma}
\begin{proof} We can assume with no loss of generality $\Minl{A}<\infty$, hence $\mm(A^s\setminus A)\to 0$ as $s\downarrow 0$.
We start from the semigroup inclusion \eqref{eq:0} to get
$$
\frac{\mm(A^{s+t})-\mm(A^s)}{t}\geq \frac{\mm((A^s)^t)-\mm(A^s)}{t}.
$$
Now, at any differentiability point $s$ of the nondecreasing and left continuous map $f(r):=\mm(A^r)$ we get $f'(s)\geq\Minu{A^s}$ and an integration w.r.t. $s$
gives
$$
\mm(A^r)-\mm(A)=f(r)-f(0_+)\geq\int_0^r f'(s)\,ds\geq \int_0^r\Minu{A^s}\, ds\qquad\forall r>0.
$$
Dividing both sides by $r$ and letting $r\downarrow 0$ an application of the mean value theorem gives the first claim.

In order to prove the equality of relaxed Minkowski contents, let $A_h\to A$ in $\mm$-measure and apply \eqref{eq:4}
to get
$$
\RelMinu{A}\leq\liminf_{h\to\infty}\RelMinu{A_h}\leq
\liminf_{h\to\infty}\Minl{A_h}.
$$
The arbitrariness of $A_h$ yields the inequality $\RelMinu{A}\leq\RelMinl{A}$.
\end{proof}

\section{Equality of relaxed contents and perimeter}

In this section we shall use the semigroup
\begin{equation}\label{eq:6}
T_t f(x):=\sup_{B_t(x)} f,\qquad T_0f=f,
\end{equation}
which extends the operator $A\mapsto A^t$ from characteristic functions to functions, i.e. $T_t\chi_A=\chi_{A^t}$ for all $t>0$.
Actually the term ``semigroup'' is not totally correct here, since only the inequality $T_{s+t}f\geq T_s(T_t f)$ holds for the operator
$T_t$ in \eqref{eq:6}; the inequality is an equality if $(X,\sfd)$ is a length space 
(see also \cite{AmDi} for the simple proof and an example of strict inequality). 
We will only need the inequality in the sequel.

Notice that $T_t f\geq f$ and that $T_t f$ is lower semicontinuous for all $t>0$. The definition of $\Slope{f}$ gives also immediately
\begin{equation}\label{eq:10}
\limsup_{t\downarrow 0}\frac{T_t f-f}{t}\leq\Slope{f}.
\end{equation}

\begin{remark}{\rm
The semigroup $T_t$ already had a role (up to a change of sign, resulting in the replacement of $\sup$ with $\inf$) 
in the proof given in \cite{AmDi} of the equality between relaxed upper gradients and ``measure'' upper gradients
in the context of the theory of $BV$ functions and sets of finite perimeter. The ``$\inf$'' semigroup can also be formally 
viewed as the limit as $p\to\infty$ of the semigroups
$$
T_t^pf(x):=\inf_{y\in X} f(y)+\frac{1}{pt^{p-1}}\sfd^p(y,x).
$$
associated to the Hamilton-Jacobi equation $\partial_t g+|\nabla g|^q/q=0$, see \cite{AGS11a, GRS}.

It is also worthwhile to mention that the semigroup $T_t$ has already been used in connection with the (anisotropic) Minkowski content in \cite{CLL}, but only in an Euclidean setting: their case can be regarded as our construction in  the metric space $(\R^d,\sfd_C,\Leb{d})$, where $\sfd_C$ is the gauge function associated to a convex set $C$.
}\fr\end{remark}

\begin{lemma}\label{lem:simone}
For all $f:X\to [0,\infty)$ Lipschitz with $\mm(\{f>0\}^h)<\infty$ for some $h>0$, one has
\begin{equation}\label{eq:9}
\int_0^\infty\Minl{\{f\geq t\}}\,dt\leq\liminf_{t\downarrow 0} \int_X \frac{T_t f-f}{t}\,d\mm\leq
\int_X \Slope{f}\,d\mm.
\end{equation}
\end{lemma}
\begin{proof}
We start from the elementary ``coarea'' pointwise identities $\int_0^\infty\chi_{\{f\geq t\}}\,dt=f$ and 
\[
\int_0^\infty T_h\chi_{\{f\geq t\}}\,dt=\int_0^\infty \chi_{\{f\geq t\}^h}\,dt=T_hf
\]
for $f:X\to [0,\infty]$. If $f$ is Borel with $\int_X f\,d\mm<\infty$, subtracting the first identity from the second one and integrating w.r.t. $\mm$ gives
$$
\int_0^\infty \frac{\mm(\{f\geq t\}^h)-\mm(\{f\geq t\})}{h}\,dt=\int \frac{T_h f-f}{h}\,d\mm.
$$
Now, notice that one has the bound
$$
\frac{T_h f-f}{h}\leq {\rm Lip}(f)\chi_{\{f>0\}^h}.
$$
Hence, assuming in addition that $f$ is Lipschitz and that  for some $h>0$ we have $\mm(\{f>0\}^h)<\infty$, Fatou's lemma, dominated convergence and \eqref{eq:10}
give \eqref{eq:9}.
\end{proof}

The following definition of perimeter, first proposed in \cite{Mi} and further investigated in \cite{AmDi} (dropping the local compactness
assumption on the metric structure) is by now well studied, see also the following remarks.

\begin{definition}[Perimeter]\label{def:Perimeter}
Let $A\subset X$ be a Borel set with $\mm(A)<\infty$. We define
\begin{equation}\label{eq:defper}
\Per{A}:=\inf\left\{\liminf_{h\to\infty}\int_X \Slope{f_h}\,d\mm:\ f_h\in {\rm Lip}(X),\,\,\,\lim_{h\to\infty}\int_X|f_h-\chi_A|\,d\mm=0\right\}.
\end{equation}
\end{definition}

If $\mm(X)$ is finite we can also consider the perimeter of $X\setminus A$; it is then easy to check that $\Per{A}=\Per{X\setminus A}$.
In the following remarks we show that smaller classes than ${\rm Lip}(X)$ can be considered in the definition of $\Per{A}$, comparing
also with the definitions in \cite{Mi} and \cite{AmDi}.

\begin{remark}\label{rem_more_general}{\rm
A simple truncation argument shows that we need only to consider sequences $f_h$ with $0\leq f_h\leq 1$ in \eqref{eq:defper}. Also, if $\bar x\in X$
and $\phi_R:X\to [0,1]$ are $1$-Lipschitz functions with $\phi_R\equiv 1$ on $B_R(\bar x)$ and $\phi_R\equiv 0$ on $X\setminus B_{2R}(\bar x)$,
from the inequality 
$$
\int_X\Slope{f_h\phi_R}\,d\mm\leq \int_X\Slope{f_h}\,d\mm+\int_{X\setminus B_R(\bar x)}f_h\,d\mm
$$
and a diagonal argument (i.e. choosing $R=R(h)$) we obtain that only sequences of Lipschitz functions
$f_h:X\to [0,1]$ with $\{f_h>0\}$ bounded need to be considered in the definition of $\Per{A}$.\fr}
\end{remark}

\begin{remark}\label{rem_equivalent_definitions}{\rm
Definition~\ref{def:Perimeter} appeared first in \cite{Mi} in complete and doubling metric measure spaces, requiring approximation
by locally Lipschitz functions and $L^1_{\rm loc}$ convergence. Then, at the general level of complete and separable metric spaces
(for instance the most appropriate setting for the theory of $CD(K,\infty)$ spaces),
the definition has been revisited in \cite{AmDi}, requiring $L^1$ convergence of $f_h$ to $\chi_A$, as in \eqref{eq:defper}, but requiring 
$f_h$ to be locally Lipschitz (i.e. for any $x\in X$ there exists $r>0$ such that $f\vert_{B_r(x)}$ is Lipschitz). As in \cite{Mi}, the motivation
for this definition is that, when localized to open sets $U\subset X$ via
$$
\mu(U):=\inf\left\{\liminf_{h\to\infty}\int_U \Slope{f_h}\,d\mm:\ f_h\in {\rm Lip}_{\rm loc}(U),\,\,\,\lim_{h\to\infty}\int_U|f_h-\chi_A|\,d\mm=0\right\}
$$
is continuous w.r.t. monotone nondecreasing sequences and thus it
provides the restriction to open sets of a finite Borel measure, namely the perimeter measure.

In proper metric spaces (a class that includes complete and doubling metric spaces) it is not hard to prove that all these variants 
of Definition~\ref{def:Perimeter} lead to the same definition, for sets with finite $\mm$-measure.
The equivalence persists also for complete and separable spaces, but the proof is not elementary, see
Section~4 and, in particular, Theorem~4.5.3 in \cite{TesiSimone}. However, the equivalence with \cite{AmDi} 
does not play any role in the paper and 
Definition~\ref{def:Perimeter} will be our working definition.
}\fr\end{remark}

\begin{theorem}\label{thm:main}
For any Borel set $A\subset X$ with $\mm(A)<\infty$ one has $\RelMinl{A}=\Per{A}$.
\end{theorem}
\begin{proof}\ {\bf Claim: $\Minl{A}\geq\Per{A}$}. We can assume with no loss
of generality $\Minl{A}<\infty$, hence $\mm(A^r\setminus A)\to 0$ as $r\downarrow 0$. 
Let $\tau\in (0,1)$ be fixed and let $f_r:X\to [0,1]$ be given by 
$$f_r(x)=1- 1\wedge\frac{\sfd_{A^s}(x)}{r'}\qquad r>0,$$
where  
$s>0$ satisfies $s+r'<r$. Since $A$ is contained in the open set $A^s$, where $f_r$ is identically equal to 1, one has $\Slope{f_r}\equiv 0$ on $A$;
on the other hand, since the complement of $A^r$ is contained in the open set $\{d_{A^s}>r'\}$, where $f_r$ is identically equal to 0,
$\Slope{f_r}\equiv 0$ on the complement of $A^r$. Thus, the inequalities $\Slope{f_r}\leq {\rm Lip}(f_r)\leq 1/r'\leq 1/(\tau r)$ give
$$
\frac{\mm(A^r)-\mm(A)}{r}\geq
\tau \int_X\Slope{f_r}\,d\mm.
$$

We now estimate
$$
\int_X|f_r-\chi_A|\,d\mm=\int_{X\setminus A}f_r\,d\mm\leq\mm(A^r\setminus A)\rightarrow 0.
$$
Then, choosing an infinitesimal sequence $(r_h)$ of radii on which the $\liminf$ is achieved, 
since $f_{r_i}\to\chi_A$ in $L^1(X,\mm)$ the very definition of $\Per{A}$ gives $\Minl{A}\geq \tau\Per{A}$.
Eventually we let $\tau\uparrow 1$.

\noindent{\bf Claim: $\RelMinl{A}\geq\Per{A}$}. Let $A_h\to A$ in $\mm$-measure and use the lower semicontinuity of the perimeter and the above claim to get
$$
\Per{A}\leq\liminf_{h\to\infty}\Per{A_h}\leq\liminf_{h\to\infty}\Minl{A_h}.
$$
Since $A_h$ are arbitrary, this proves that $\Per{A}\leq\RelMinl{A}$.

\noindent{\bf Claim: $\Per{A}\geq\RelMinl{A}$}. Thanks to Remark~\ref{rem_more_general} we can find a family of Lipschitz
functions $f_h:X\to [0,1]$ with $\{f_h>0\}$ bounded, $f_h\to\chi_A$ in $L^1(X,\mm)$ and 
$\int_X\Slope{f_h}\,d\mm\to \Per{A}$. For all $\epsilon\in (0,1/2)$ we can find, thanks to \eqref{eq:9},
$t_h\in (\epsilon,1-\epsilon)$ with
$$
\Minl{\{f_h\geq t_h\}}\leq\frac{1}{1-2\epsilon}\int_X\Slope{f_h}\,d\mm.
$$
Since $\{f_h\geq t_h\}\to A$ in $\mm$-measure 
it follows that $(1-2\epsilon)\RelMinl{A}\leq\Per{A}$ and letting $\epsilon\downarrow 0$ the inequality is achieved.
\end{proof}

\section{Level sets of Lipschitz functions}

In this section we study the relation between perimeter and Minkowski content for generic superlevel sets of Lipschitz
functions.

Definition~\ref{def:Perimeter} of perimeter is a particular case of the following one, for $f\in L^1(X,\mm)$:
\[
\Var{f}:=\inf\left\{\liminf_{h\to\infty}\int_X \Slope{f_h}\,d\mm:\ f_h\in {\rm Lip}(X),\,\,\,\lim_{h\to\infty}\int_X|f_h-f|\,d\mm=0\right\}.
\]
The two concepts are closely related. Indeed, by approximating any $L^1$ function with step functions and, conversely, characteristic functions
of $\{f>t\}$ by $\chi_\eps\circ f$, where $\chi_\epsilon$ is a smooth approximation of $\chi_{[t,\infty)}$, it can be easily proved (see
\cite{Mi} for details) that the classical coarea formula of Fleming-Rishel, namely
\begin{equation}\label{eq:12}
\Var{f}=\int_0^\infty\Per{\{f\geq t\}}\,dt\qquad\text{$f\in L^1(X,\mm)$, $f\geq 0$}
\end{equation}
holds even in this abstract setting, without any finiteness assumption on either side of the equality. Under the regularity
assumption 
\begin{equation}\label{eq:13}
\Var{f}=\int_X\Slope{f}\,d\mm\qquad\text{for all $f\in L^1(X,\mm)$ Lipschitz}
\end{equation}
on the metric measure structure
we can now prove that we can replace perimeter with the lower Minkowski content in \eqref{eq:12}, namely
\begin{equation}\label{eq:12bis}
\Var{f}=\int_0^\infty\Minl{\{f\geq t\}}\,dt\qquad\text{$f\in L^1(X,\mm)$, $f\geq 0$.}
\end{equation}

\begin{remark}\label{It_is_not_trivial}{\rm
In Euclidean spaces $\R^n$, combining the Fleming-Rishel formula, the pointwise inequalities
$$
\Per{\{f\geq t\}}=\Haus{n-1}(\partial^*\{f\geq t\})\leq
\Haus{n-1}(\partial\{f>t\})\leq\Haus{n-1}(\{f=t\})
$$
(where $\partial^*\{f\geq t\}$ is the essential boundary of $\{f\geq t\}$, a countably $\Haus{n-1}$-rectifiable set whenever $\Per{\{f\geq t\}}$ is finite)
and the coarea formula for nonnegative Lipschitz functions
$$
\int_{\R^n}|\nabla f|\,dx=\int_0^\infty \Haus{n-1}(\{f=t\})\,dt
$$
one obtains
$$
\Per{\{f\geq t\}}=\Haus{n-1}(\partial\{f\geq t\})=\Haus{n-1}(\{f=t\})\qquad\text{for $\Leb{1}$-a.e. $t>0$}
$$
and that $\{f=t\}$ is countably $\Haus{n-1}$-rectifiable for $\Leb{1}$-a.e. $t>0$.
However, the validity of \eqref{eq:12bis} is a nontrivial information
even in the Euclidean case, since rectifiability does not imply, in general, finiteness of the Minkowski content and agreement
with the Hausdorff measures. See \S2.13 of \cite{AFP} for an example of compact countable set having infinite Minkowski
content. Finiteness of the Minkowski contents is ensured for instance by density lower bounds of the form
$$
\sigma(B_r(x))\geq \frac{\mm(B_r(x))}{r}\qquad\forall x\in\partial A,\,\,r\in (0,1) 
$$
for some finite measure $\sigma$, see Theorem~2.104 for the simple argument, which also proves that this property, in conjunction with
rectifiability, provides agreement of the Minkowski contents with the Hausdorff measures.}\fr
\end{remark}

Notice that, thanks to the truncation
argument of Remark~\ref{rem_more_general}, an equivalent formulation of \eqref{eq:13} is the lower semicontinuity of the functional 
$f\mapsto\int_X\Slope{f}\,d\mm$, restricted to Lipschitz and integrable functions, w.r.t. $L^1(X,\mm)$ convergence. 

\begin{proposition}\label{prop:coarea}
Under assumption \eqref{eq:13}, one has
\begin{equation}\label{eq:14}
\Var{f}=\int_0^\infty\Minl{\{f\geq t\}}\,dt
\end{equation}
for any Lipschitz function $f:X\to [0,\infty)$ with $\mm(\{f>0\}^h)<\infty$ for some $h>0$.
As a consequence 
$$
\Minl{\{f\geq t\}}=\Per{\{f\geq t\}}\qquad\text{for $\Leb{1}$-a.e. $t>0$.}$$
In addition, if $f:X\to [0,M]$ is $1$-Lipschitz, $\mm(\{f>0\}^h)<\infty$ for some $h>0$ 
and $\Slope{f}=1$ $\mm$-a.e. in $\{0<f<M\}$, one has
$$
\Minl{\{f\geq t\}}=\Minu{\{f\geq t\}}\qquad\text{for $\Leb{1}$-a.e. $t\in (0,M)$.}
$$
\end{proposition}
\begin{proof} Taking \eqref{eq:12} and the inequality $\Minl{A}\geq\Per{A}$ into account, for the first part of the statement 
it is sufficient to prove the inequality $\geq$ in \eqref{eq:14}. Since by assumption $\Var{f}\geq\int_X\Slope{f}\,d\mm$, the inequality
follows at once from \eqref{eq:9} of Lemma~\ref{lem:simone}.

In order to prove the second statement we start from the inequality
$$
\frac{\mm(\{f\geq t\}^h)-\mm(\{f\geq t\})}{h}\leq\frac{\mm(\{f\geq t-h\})-\mm(\{f\geq t\})}{h}
$$
to get $\Minu{\{f\geq t\}}\leq -\mm(\{f\geq t\})'$ for $\Leb{1}$-a.e. $t>0$. By integration in $(0,M)$ we get
$$
\int_0^\infty\Minu{\{f\geq t\}}\,dt\leq\mm(\{0<f<M\})\leq\int_X\Slope{f}\,d\mm=
\int_0^\infty\Minl{\{f\geq t\}}\,dt,
$$
whence the result follows.
\end{proof}

\begin{remark}[The case of distance functions]{\rm
Assume that $(X,\sfd)$ satisfies the following length assumption: for any $x,\,y\in X$ there exists an $\epsilon$-geodesic between 
$x$ and $y$, namely a finite collection of points $z_0,\ldots,z_N$ with $z_0=x$, $z_N=y$, $\max_{0\leq i<N}\sfd(z_i,z_{i+1})<\eps$ and
$$
\sum_{i=0}^{N-1}\sfd(z_i,z_{i+1})<\sfd(x,y)+\eps.
$$
If $A\subset X$ is nonempty and closed, under this length assumption on $(X,\sfd)$ it is easy to check that
$\Slope{\sfd_A}\equiv 1$ in $X\setminus A$. Therefore, the second part of Proposition~\ref{prop:coarea} is applicable to all functions
$(M-\sfd_A)^+$ for all $M>0$ and provides, still under the assumption \eqref{eq:13}, the property
\begin{equation}\label{eq:latem1}
\Minl{\{\sfd_A\leq t\}}=\Minu{\{\sfd_A\leq t\}}=\Per{\{\sfd_A\leq t\}}%=\Minl{\{\sfd_A<t\}}=\Minu{\{\sfd_A>t\}}
\end{equation}
for $\Leb{1}$-a.e. $t>0$. If $\mm(X)<\infty$ we can apply the statement also to $\sfd_A$ to get
\begin{equation}\label{eq:latem2}
\Minl{\{\sfd_A\geq t\}}=\Minu{\{\sfd_A\geq t\}}=\Per{\{\sfd_A\geq t\}}%=\Minl{\{\sfd_A<t\}}=\Minu{\{\sfd_A>t\}}
\end{equation}
for $\Leb{1}$-a.e. $t>0$. Since $\Per{\{\sfd_A\geq t\}}=\Per{\{\sfd_A\leq t\}}$ with at most countably many exceptions,
the quantities in \eqref{eq:latem1} and \eqref{eq:latem2} coincide $\Leb{1}$-a.e. in $(0,\infty)$. 

If the length assumption above is enforced to
$$
\text{$\forall x,\,y\in X$, $\forall\eps>0$ there exists $\gamma\in C([0,1];X)$ with $\gamma_0=x$, $\gamma_1=y$, ${\rm length}(\gamma)<\sfd(x,y)+\eps$}
$$
(notice that the two length properties are equivalent in separable spaces), a simple continuity argument gives
$$
\{\sfd_A=t\}^r\setminus\{\sfd_A=t\}=\{0<\sfd_{\{\sfd_A\geq t\}}<r\}\cup\{0<\sfd_{\{\sfd_A\leq t\}}<r\}\qquad\forall t\in (0,\infty),\,\,r>0. 
$$
Then, since $\mm(\{\sfd_A=t\})=0$ with at most countably many exceptions, 
we obtain existence of the Minkowski content of the level set $\{\sfd_A=t\}$
$$
\lim_{r\downarrow 0}\frac{\mm(\{\sfd_A=t\}^r)}{2r}=\Per{\{\sfd_A\geq t\}}=\Per{\{\sfd_A\leq t\}}
$$
for $\Leb{1}$-a.e. $t>0$.}\fr
\end{remark}

Concerning assumption \eqref{eq:13}, we point that it holds on doubling $\RCD(K,\infty)$ spaces (see  \cite{AGS11b}) for the definition of the latter. Indeed, on one side we know by Rajala's paper \cite{Ra} that on these spaces a 1-1 weak Poincar\'e inequality holds, so that thanks to Cheeger's results in \cite{Chee} we have that the $p$-weak gradient is $\mm$-a.e. equal to the local Lipschitz constant of Lipschitz functions for $p>1$. On the other hand, in \cite{GiHa} it has been proved that on $\RCD(K,\infty)$ spaces the $p$-weak gradients all coincide even for $p=1$ (this result is based on the Bakry-Emery estimate in the form established by Savar\'e in \cite{Sa}), thus giving the claim.
%Let us now focus on assumption \eqref{eq:13}. 
%In the class of $\RCD(K,\infty)$ metric measure spaces introduced in \cite{AGS11b}, using the $L^\infty$-Lip regularization 
%$$
%{\rm Lip}(H_t f)\leq c(t,K)\|f\|_\infty\qquad\forall t>0
%$$
%provided by the heat semigroup $H_t$, assumption \eqref{eq:13} can be proved, see \cite{GiHa}. On the other
%hand, 
We illustrate in the next remark the known relations between $\Slope{f}$ and other weak notions of gradient, and why
doubling\&1-Poincar\'e alone are not sufficient to ensure the validity of \eqref{eq:13}.

\begin{remark}{\rm Recall that one of the key results of the seminal paper \cite{Chee} is the lower semicontinuity of
the functional $\int_X\Slope{f}^p\,d\mm$ w.r.t. $L^p$ convergence on locally Lipschitz functions, for all $p\in (1,\infty)$, assuming the
doubling property of the metric measure structure and the validity of a $p$-Poincar\'e inequality. This result can also be rephrased by
saying that $\Slope{f}$ coincides $\mm$-a.e. with the minimal $p$-weak upper gradient.

However, under the same structural
assumptions, the result is not true in general when $p=1$.
Indeed, revisiting an example by Carbone-Sbordone \cite{CaSb},  in \cite{HKLL14}, a metric measure space $(X,\sfd,\mm)$ is built as follows:
$X=[0,1]$ is endowed with the Euclidean distance and with the weighted Lebesgue measure $\mm=\omega\Leb{1}$, with $\omega\equiv 1$
on a ``fat'' Cantor set $K\subset (0,1)$ (i.e. a compact totally disconnected set with positive Lebesgue measure) and $\omega=1/2$, say, on
$(0,1)\setminus K$. It is clear that $(X,\sfd,\mm)$, being comparable to the standard Euclidean structure, satisfies the
doubling and $1$-Poincar\'e assumptions. On the other hand, in \cite{HKLL14} Lipschitz functions $f_n$ convergent to the 
identity function $f$ in $L^1(X,\mm)$ are built in such a way that
$$
\limsup_{n\to\infty}\int_X\Slope{f_n}\,d\mm<\int_X\Slope{f}\,d\mm, 
$$
thus disproving the lower semicontinuity. This corresponds to the fact that two notions of minimal $1$-weak upper gradients
are possible in the $H^{1,1}$ (and $BV$) theory. The two notions do not coincide in general, and $\Slope{f}$ coincides only with one of 
them (see \cite[Theorem~12.5.1]{HKST}).
See also \cite{AmPiSp} and \cite{HKLL14} for a more detailed discussion.\fr}
\end{remark}

\section{Some open problems}

In complete metric measure spaces with a doubling measure, under the assumption of the validity of a 1-Poincar\'e inequality,
the first author proved in \cite{ambrosio-ahlfors}, \cite{ambrosio-doubling} that the perimeter measure of Remark~\ref{rem_equivalent_definitions}
coincides con $\theta\Hausph{h}\res\partial^* E$, where $\partial ^*E$ is the essential boundary of $E$
(i.e. the complement of the union of density and rarefaction points of $E$) and $\Hausph{h}$ is the measure built of out the
gauge function $\zeta(\overline{B}_r(x))=\mm(\overline{B}_r(x))/(2r)$ with Carath\'eodory's construction. The density $\theta$
is bounded from below by the structural constants involved in the doubling and Poincar\'e assumptions, but little more is known
in general about it. Under additional regularity assumptions, it would be interesting to relate more closely the Minkowski content of 
$E$ with the measure $\Hausph{h}\res\partial^* E$, as in the Euclidean-Riemannian theory.

In the same vein, one can prove the coarea inequality (Proposition~5.1 in \cite{ambrosio-doubling}) 
\begin{equation}\label{eq:alba}
\int_0^\infty\Hausph{h}(B\cap\{f=t\})\,dt\leq {\rm Lip}(f)\mm(B)  \qquad\text{$B$ Borel}
\end{equation}
reminiscent of \eqref{eq:14}: we propose here a self-contained proof of an improved version of this inequality, involving the asymptotic
Lipschitz constant
$$
{\rm Lip}_a(f,x):=\lim_{r\downarrow 0}{\rm Lip}(f,B_r(x))
$$
and, with an additional factor 2, the slope.

\begin{proposition} Assume that $\mm$ is a doubling measure in the 
metric space $(X,\sfd)$.
Let $f:X \to [0, \infty)$ be a Lipschitz function with $\mm(\{ f>0 \}) < \infty$. Then for every $B\subset X$ Borel one has
\begin{equation}\label{eq:alba2}
\int_0^\infty\Hausph{h}(B\cap\{f=t\})\,dt\leq  \int_B {\rm Lip}_a (f,\cdot) \, d \mm,
\end{equation}
\begin{equation}\label{eq:alba3}
\int_0^\infty\Hausph{h}(B\cap\{f=t\})\,dt\leq  2 \int_B  \Slope{f}  \, d \mm.
\end{equation}
\end{proposition}

\begin{proof} Let $B\subset X$ be a bounded Borel set. First we recall that if $\mm$ is a doubling measure then these two properties hold, which are both consequences of the Vitali covering theorem w.r.t.
doubling measures:
\begin{itemize}
\item[(i)] for all $\delta>0$, $\mm(B)= \inf \left\{ \sum_i \mm(B_i) \; : \; B_i=\overline{B}_{r_i}(x_i) ,\,  r_i\in (0,\delta), \, B \subset \bigcup B_i, \right\}$;
\item[(ii)] for every collection of closed balls $\mathcal{F}$ that is a fine cover of $B$ we have a disjoint collection $\mathcal{F}'\subset{\mathcal F}$ that covers $\mm$-almost all of $B$.
\end{itemize}

For every $\delta>0$ and every $\varepsilon>0$ we have, by property (i), the existence of closed balls $B_i$ such that $B \subset\bigcup B_i$, 
radius $r_i<\delta/2$ and $\mm(B) \leq \varepsilon + \sum_i \mm (B_i)$. Let us denote $t_i^- = \inf_{B_i} f$ and $t_i^+ = \sup_{B_i} f$: then $B_i \cap \{ f = t \} \neq \emptyset $ implies 
that $t \in [t_i^- , t_i^+]$ and in particular we have 
$$\Hausph{h}_\delta ( B \cap \{ f=t\} ) \leq \sum_{i \text{ s.t. }  t \in [t_i^- , t_i^+]}\frac  {\mm(B_i) } { 2r_i  }. $$
Integrating this inequality from $0$ to $\infty$ we obtain
\begin{equation}\label{eqn:discrete} \int_0^{\infty} \Hausph{h}_{\delta} ( B \cap \{ f=t\} ) \, dt \leq 
\sum_i \frac  {\mm(B_i) } { 2r_i  }(t_i^+ - t_i^- ). \end{equation}
It is clear that $(t_i^+ - t_i^-) \leq 2r_i {\rm Lip }(f)$, and thus the countable subadditivity of $\Hausph{h}_\delta$ gives 
$$ \int_0^\infty \Hausph{h}_\delta ( B \cap \{ f=t\} ) \, dt \leq {\rm Lip }(f) \sum_i \mm(B_i) \leq {\rm Lip }(f) (\mm(B) + \varepsilon ). $$
Letting $\varepsilon,\ , \delta \to 0$ we will get \eqref{eq:alba}. This proves also that $ \Hausph{h} ( B \cap \{ f=t\} ) =0$ for $\Leb{1}$-a.e. $t\in (0,\infty)$
whenever $\mm(B)=0$; in this case, a fortiori, for every $\delta>0$ one has $ \Hausph{h}_\delta ( B \cap \{ f=t\} ) =0$ for $\Leb{1}$-a.e. $t\in (0,\infty)$.
 
Now we use property (ii): for every $\delta$ we find a a disjoint family of closed balls $\{B_i\}_{i \in \mathbb{N}}$ with radii less than $\delta/2$ such that 
$\mm( B \setminus \tilde{B})=0$, where $\tilde{B} = \bigcup_i B_i$. It follows that
 $$ \Hausph{h}_\delta ( B \cap \{ f=t\} ) \leq \Hausph{h}_\delta ( \tilde{B} \cap \{ f=t\} ) + 
 \Hausph{h}_{\delta} ( (B \setminus \tilde{B} ) \cap \{ f=t\} )   = \Hausph{h}_{\delta} ( \tilde{B} \cap \{ f=t\} ) $$
 for $\Leb{1}$-a.e. $t\in (0,\infty)$ 
and thus $\int_0^{\infty}  \Hausph{h}_\delta ( B \cap \{ f=t\} ) \, dt = \int_0^{\infty}  \Hausph{h}_{\delta} ( \tilde B \cap \{ f=t\} ) \, dt$. 
Now we can apply \eqref{eqn:discrete} for $\tilde B$ and its covering to obtain the same inequality for $B$. Now we notice now that for all
$x \in B_i$ one has
\begin{align*} \frac{ t_i^+ - t_i^- } { 2r_i } &= \sup_{ y , \,y' \in B_i } \frac { f(y) - f(y') } { 2r_i } 
  \leq \sup_{ y ,\, y' \in B_i } \frac { |f(y)-f(x)| + |f(y')-f(x)| } { 2r_i }\\
&  \leq 2 \sup_{0< \sfd ( x,y) \leq \delta } \frac{ |f(y) - f(x)| } { \sfd (y,x) } =: 2 {\rm sl} ( x,  \delta), 
\end{align*}
where ${\rm sl}(\cdot,\delta)$ is the local slope on scale $\delta$, satisfying ${\rm sl } (x, \delta)\downarrow \Slope{f}(x)$ as $\delta\downarrow 0$. 
In particular, using that $B_i$ are disjoint and cover $\mm$-almost all of $B$, we have deduce from \eqref{eqn:discrete} and the last inequality that 
\begin{equation}\label{eqn:delta} \int_0^{\infty} \Hausph{h}_{\delta} ( B \cap \{ f=t\} ) \, dt \leq  2 \int_B  {\rm sl} (x,  \delta) \, d \mm. \end{equation}
Now, we conclude letting $\delta \downarrow 0$ and using dominated convergence, since ${\rm sl}(\cdot,\delta)\leq {\rm Lip}(f)$. 
We can obtain the analogous result with ${\rm Lip}_a$ in the right hand side using that, for $x \in B_i$,
$$\frac{ t_i^+ - t_i^- } { 2r_i } \leq {\rm Lip}(f, B_i) \leq {\rm Lip}(f, B_{2\delta}(x)) \to {\rm  Lip}_a (f,x).$$
\end{proof}

\begin{corollary} Let $(X,\sfd,\mm)$ be a doubling metric measure space, with $(X,\sfd)$ complete, such that \eqref{eq:13} holds. 
Then for every Lipschitz function $f:X \to [ 0, \infty)$  such that $\mm(\{ f >0\})< \infty$ 
the set $E_t = \{ f \geq t\}$ is of finite perimeter and $P(E_t, \cdot ) \geq \frac 12 \Hausph{h} \res \partial^* E_t$
for $\Leb{1}$-a.e. $t\in (0,\infty)$.
\end{corollary}

It would be nice to improve \eqref{eq:alba}, replacing ${\rm Lip}(f)\mm(B)$ with $\int_B\Slope{f}\,d\mm$ in the right hand side,
(compare with \eqref{eq:alba3}), but this seems difficult to obtain without extra assumptions. We notice also that anyhow this inequality is not optimal since, in the Euclidean case $X=\R^d$, we have $\Hausph{h} = \frac{ \omega_{d}}{2 \omega_{d-1}} \Haus{d-1}$
on rectifiable sets.

Finally, a natural question is the identification of regularity assumptions ensuring equality between Minkowski content and perimeter, somehow
bypassing the question of connecting these notions to $\Hausph{h}\res\partial^* E$.

\end{document}